\documentclass[12pt]{amsart}

\usepackage{bm}
\usepackage{amsfonts,amssymb}
\usepackage{dsfont}
\usepackage{amscd}
\usepackage{extarrows}
\usepackage{amsmath}
\usepackage{enumerate}
\usepackage{amscd}
\usepackage[all]{xy}
\usepackage[hyperfootnotes=true]{hyperref}
\usepackage{mathrsfs}
\usepackage{hyperref}
\theoremstyle{fancy}
\newtheorem{thm}{Theorem}[section]

\newtheorem{lemma}[thm]{Lemma}

\newtheorem{definition}{Definition}[section]
\newtheorem{conj}{Conjecture}[section]

\newtheorem{rmk}{Remark}[section]

\begin{document}

\title{On Grauert-Riemenschneider type criterions}
\author{Zhiwei Wang}
\address{ School
of Mathematical Sciences\\Beijing Normal University\\ Beijing\\ 100875\\ P. R. China}
\email{zhiwei@bnu.edu.cn}

%\author[H. Fan]{Huijun Fan}
%\address{School
%of Mathematical Sciences\\Peking University\\ Beijing\\ 100871\\ China}
%\email{fanhj@math.pku.edu.cn}



%    author one information

\begin{abstract}Let $(X,\omega)$ be a compact Hermitian manifold of complex dimension $n$. In this article, we first survey  recent progress towards Grauert-Riemenschneider type criterions.  Secondly, we give a simplified proof of Boucksom's conjecture  given by the author under the assumption that the  Hermitian metric $\omega$ satisfies $\partial\overline{\partial}\omega^l=$ for all $l$, i.e., if $T$ is a closed positive current on $X$ such that $\int_XT_{ac}^n>0$, then the class $\{T\}$ is big and $X$ is K\"{a}hler. Finally, as an easy observation, we point out that Nguyen's result can be generalized as follows: if $\partial\overline{\partial}\omega=0$, and $T$ is a closed positive current with analytic singularities, such that $\int_XT^n_{ac}>0$, then the class $\{T\}$ is big and $X$ is K\"{a}hler.
\end{abstract}

\thanks{The author was partially supported by the Fundamental Research Funds for the Central Universities and by the NSFC grant NSFC-11701031}

\keywords{Closed positive current, Demailly-P\u{a}un's conjecture,  Boucksom's conjecture, K\"{a}hler current, Fujiki class}

\maketitle

\section{Introduction}
Let $(X,\omega)$ be a compact Hermitian manifold of dimension $n$, $\omega$ be a Hermitian metric on $X$. The Grauert-Riemenschneider conjecture (G-R conjecture for short)  \cite{GR70} reads that: if there is a holomorphic Hermitian line bundle $(L,h)\rightarrow X$, such that $c(L)=\frac{i}{2\pi}\Theta_h\geq 0$ on $X$ and is strictly positive on a dense open subset of $X$, then $X$ is Moishezon. A compact complex manifold is said to be Moishezon, if it is birational to a projective manifold. 

 This conjecture was first proved by Siu \cite{Siu84b} and then by Demailly \cite{Dem85} shortly later. 

 Siu even  proved that if $c(L)\geq 0$ on $X$ and strictly positive at least at one point in $X$, then $X$ is Moishezon.  Siu first proved that, under the curvature condition of G-R conjecture, one has 
 \begin{align*}
 \dim H^q(X,L^k)=o(k^n), ~~~~\mbox{for}~~~~q\geq 1.
 \end{align*}
Then, by a Riemann-Roch theorem argument, one can get that the Kodaira dimension of  $L$ is of top dimension, i.e., dimension $n$. Thus $X$ is Moishezon.

 Demailly \cite{Dem85, Dem91} established the celebrated holomorphic Morse inequalities. The strong version of holomorphic Morse inequalities is stated as follows:
 \begin{align*}
 \sum_{0\leq j\leq q}(-1)^{q-j}\dim_\mathbb{C} H^j(X,L^k)\leq (-1)^q \frac{k^n}{n!}\int_{X(\leq q)}(c(L))^n+o(k^n)， 
 \end{align*}
 where $X(\leq q) $ be the set where $c(L)$ is  non degenerate and has at most q negative eigenvalue. Take $q=1$ in above inequalities, one can get that 
 \begin{align*}
 \dim H^0(X,L^k)&\geq \dim H^0(X,L^k)-\dim H^1(X,L^k)\\
 &\geq \frac{k^n}{n!} \int_{X(\leq 1)}(c(L))^n+o(k^n).
 \end{align*}
Under the curvature condition of G-R conjecture, one finds $X(\leq 1)$ is a non-empty open set of $X$, then the integral term $ \int_{X(\leq 1)}(c(L))^n>0$, from which one can easily find the Kodaira dimension of $L$ is $n$, i.e., $X$ is Moishezon.  The advantage of Demailly's holomorphic Morse inequalities is that they give a criterion of a Hermitian holomorphic line bundle to be big by the positivity of an integral of the curvature, which is a little bit weaker than the positivity of the curvature itself.

Around 1990's, due to the development of several complex variables,  singular metrics were introduced to holomorphic line bundles.  An Hermitian metric $h$ of a holomorphic line bundle $L$ is said to be singular, if locally
we can write $h=e^{-2\varphi}$, with $\varphi\in L^1_{loc}$.   There are many positivities due to the curvature current of the singular metric, e.g., semi positive, nef, pseudo effective and big. In general, one has semi positive $\subset$ nef $\subset $ pseudo effecitive $\subset $ big.

In \cite{JS93}, Ji-Shiffman proved  that a holomorphic line bundle is big if and only if there is a singular metric of $L$ such that the curvature current is positive definite, i.e., is a K\"{a}hler current.

A compact complex manifold is said to be in the Fujici class $\mathcal{C}$, if  it is bimeromorphic to a compact K\"{a}hler manifold. It was proved by Demailly-Paun in \cite{DP04} that $X$ is in the Fujiki class $\mathcal{C}$ if and only if it carries a K\"{a}hler current.

Recently, experts in several complex variables and complex geometer try to relax the assumptions in G-R conjecture, say from  the integral class to the transcendental class, and from semi positivities to nefness and pseudo effectiveness.
Along these directions, there are the following conjectures.

\begin{conj}[Boucksom's conjecture, cf. \cite{Bou02}]\label{conj}
If a compact complex manifold $X$ carries a closed positive $(1,1)$-current $T$ with $\int_X T^n_{ac}>0$, then the class $\{T\}$$X$ is big,  and thus $X$ in the  \textit{Fujiki} class $\mathcal{C}$.
\end{conj}

\begin{conj}[Demailly-P\u{a}un's conjecture, cf. \cite{DP04}]\label{DP conj}
	Let $X$ be a compact complex manifold of complex dimension $n$ and $\alpha\in H^{1,1}_{\partial\overline{\partial}}(X,\mathbb{R})$ be  a nef class, if $\int_X\alpha^n>0$, then $\alpha$ is big and thus  $X$ is in the Fujiki class $\mathcal{C}$.
\end{conj}

	Recently, there has been progress of the above conjectures.
	
	If the class  $\{T\}$ is assumed to be integral (i.e. $T$ is a curvature current associated to a singular metric of  a pseudo-effective holomorphic line bundle), the conjecture \ref{conj} is proved by Popovici \cite{Pop08} by using singular Holomorphic Morse inequalities of Bonavero \cite{Bon98} and an improved regularization of currents with mass control. But in general case, there is no holomorphic Morse type inequalities for real  transcendental  $(1,1)$-class.
	For this reason, Boucksom-Demailly-P\u{a}un-Peternell posed the following 
	\begin{conj}[Boucksom-Demailly-P\u{a}un-Peternell's conjecture, cf. \cite{BDPP}] 	Let $X$ be a compact complex manifold of complex dimension $n$. Let $\alpha$ be  a  closed real $(1,1)$-form and let $X(\alpha,\leq 1) $ be the set where $\alpha$ non degenerate and has at most one negative eigenvalue. If $\int_{X(\alpha,\leq 1)}\alpha^n>0$, then the class $\{\alpha\}\in H^{1,1}_{\partial\overline{\partial}}(X,\mathbb{R})$ is big and 
		\begin{align*}
		vol(\{\alpha\}):=\sup_{0<T\in \{\alpha\}}\int_{X\setminus Sing(T)}T^n\geq \int_{X(\alpha,\leq 1)}\alpha^n,
		\end{align*}
		where $T$ ranges over all K\"{a}hler currents $T\in \{\alpha\}$ with analytic singularities.
	\end{conj}

 If $X $ is assumed to be compact K\"{a}hler,  then 
\begin{itemize}
	\item Conjecture \ref{DP conj} is proved by Demailly-P\u{a}un in \cite{DP04}. 
	\item Conjecture \ref{conj}  is proved by Boucksom in  \cite{Bou02}, i.e.   $\alpha$ is big if and only if vol$(\alpha)>0$.  It is worth to mention that if $\alpha$ is nef, then vol$(\alpha)=\int_X\alpha^n$, thus Boucksom's proof of Conjecture \ref{conj} yields a proof of Conjecture \ref{DP conj}.  
	\end{itemize}

 If $(X,\omega)$ is only assumed to be compact Hermitian, then 
 \begin{itemize}
\item Under the assumption that $\partial\overline{\partial}\omega^l=0$ for all $l$, Chiose gives a  proof of Conjecture \ref{DP conj} in \cite{Chi16}  by combining Lamari's criterion \cite{Lam99a} and Tosatti-Weinkove's solution to complex Monge-Amp\`{e}re equation on compact Hermitian manifolds \cite{TW10}. See also \cite{Tos16} for even more simplified proof in this case. Note that
there always exists such a metric on any compact complex surface.

\item  Under the assumption that $\partial\overline{\partial}\omega=0$, and the class $\alpha$ is semi positive, Nguyen proved both Conjecture \ref{conj} and  Conjecture \ref{DP conj} in \cite{N16} by combining Lamari's criterion and analysis of  degenerate complex Monge-Amp\`{e}re equations.
\item Under the assumption that  $\partial\overline{\partial}\omega^l=0$ for all $l$, the author gives a proof of  Conjecture \ref{conj}  by modifying Chiose's arguments. 
\end{itemize}

All the above positive results are called \textbf{Grauert-Reimenschneider type criterion}. For the convinient of the reader, we state theorems of Nguyen and the author as follows.

\begin{thm}[\cite{N16}]\label{Ngu} Suppose that  $X$ is a compact complex manifold equipped with a pluriclosed  metric $\omega$,  i.e. $\partial\overline{\partial}\omega=0$, $\{\beta\}$ is a semi-positive class of type $(1,1)$, such that $\int_X\beta^n>0$, then $\{\beta\}$ contains a  K\"{a}hler current. Thus this gives a partial solution to the Conjecture \ref{DP conj}.
\end{thm}

\begin{thm}[\cite{W16}]\label{Wang} Suppose that $X$ is a compact complex manifold equipped with a Hermitian metric $\omega$ satisfying $\partial\overline{\partial}\omega^k=0$ for all $k$, the volume vol$(\alpha)$ of a pseudo-effective class $\alpha\in H^{1,1}_{\partial\overline{\partial}}(X,\mathbb{R})$ is defined and proved to be finite. Then it is proved that if $\alpha$ is pseudo-effective and   vol$(\alpha)>0$, then  $\alpha$ is big, thus $X$ is in the Fujiki class $\mathcal{C}$, and finally K\"{a}hler. This gives a partial solution to the Conjecture \ref{conj}.
\end{thm}
\begin{rmk}
	It is worth to mention that Theorem \ref{Wang}  drops the assumption of nefness of the class $\alpha$, which is an  improvement of Chiose's result, since in general the nef cone of a compact Hermitian manifold is a subset of the pseudo effective cone.
\end{rmk}

In this paper, absorbing new techniques towards the above mentioned conjectures, we first give a simplified  proof of Theorem \ref{Wang}.  Secondly, as an easy observation,  we point out that  Nguyen's results can be generalized to the following

\begin{thm}\label{main theorem}
Let $(X,\omega)$ be a compact complex manifold of complex dimension $n$, $\omega$ be a Hermitian metric with $\partial\overline{\partial}\omega=0$,  and $T$ be a closed positive current with analytic singularities (see Definition \ref{APC}), if $\int_X T^n_{ac}>0$, then $\{T\}$ is big. Thus $X$ is in the Fujiki class $\mathcal{C}$, and finally K\"{a}hler. Where $T_{ac}$ is the absolutely continuous part in the Lebesgue decomposition of $T$ with respect to the Lebesgue measure on $X$.
\end{thm}
%\begin{rmk} Theorem \ref{main theorem} is a partial answer to the Conjecture \ref{conj}. A full answer of Boucksom's conjecture and Demailly-P\u{a}un's conjecture may rely highly on the analysis of complex Monge-Amp\`{e}re type equation on compact Hermitian manifolds.
%\end{rmk}

%Thirdly,  we discuss how to remove the smoothness   assumption on $\beta$ in Theorem \ref{Ngu}.

\subsection*{Acknowledgements}
The author is  deeply grateful to  Prof. Xiangyu Zhou for his constant support and  encouragement over the years.

\section{Preliminaries}

\subsection{$\partial\overline{\partial}$-cohomology}\label{BC-cohomology}
Let $X$ be an arbitrary compact complex manifold of complex dimension $n$. Since the $\partial\overline{\partial}$-lemma does not hold in general, it is better to work with $\partial\overline{\partial}$-cohomology which is defined as
\begin{align}
H^{p,q}_{\partial\overline{\partial}}(X,\mathbb{C})=\big (\mathcal{C}^\infty(X,\Lambda^{p,q}T^*_X)\cap \ker d)/\partial\overline{\partial}\mathcal{C}^\infty(X,\Lambda^{p-1,q-1}T^*_X).\notag
\end{align}

By means of  the Fr\"{o}licher spectral sequence, one can see  that  $H^{p,q}_{\partial\overline{\partial}}(X,\mathbb{C})$ is  finite dimensional and can be computed either with spaces of smooth forms or with currents. In both cases, the quotient topology of $H^{p,q}_{\partial\overline{\partial}}(X,\mathbb{C})$ induced by the Fr\'{e}chet topology of smooth forms or by the weak topology of currents is Hausdorff, and the quotient map under this Hausdorff topology is continuous and open.

In this paper, we will just need the $(1,1)$-cohomology space $H^{1,1}_{\partial\overline{\partial}}(X,\mathbb{C})$.  The real structure on the space of $(1,1)$-smooth forms (or $(1,1)$-currents) induces a real structure on $H^{1,1}_{\partial\overline{\partial}}(X,\mathbb{C})$, and we denote by $H^{1,1}_{\partial\overline{\partial}}(X,\mathbb{R})$ the space of real points.  A class $\alpha\in H^{1,1}_{\partial\overline{\partial}}(X,\mathbb{C})$ can be seen as an affine space of closed $(1,1)$-currents. We denote by $\{T\}\in H^{1,1}_{\partial\overline{\partial}}(X,\mathbb{C}) $ the class of the current $T$. Since $i\partial\overline{\partial}$ is a real operator (on forms or currents), if $T$ is a real closed $(1,1)$-current, its class $\{T\}$ lies in $H^{1,1}_{\partial\overline{\partial}}(X,\mathbb{R})$ and consists of all the closed currents $T+i\partial\overline{\partial}\varphi$ where $\varphi$ is a real current of degree $0$.

\begin{definition}\label{pf-nef-kahler}
Let $(X,\omega$) be a compact Hermitian manifold. A cohomology class $\alpha\in H^{1,1}_{\partial\overline{\partial}}(X,\mathbb{R})$ is said to be pseudo-effective  iff it contains a positive current; $\alpha$ is nef  iff, for each $\varepsilon>0$, $\alpha$ contains a smooth form $\theta_\varepsilon$ with $\theta_\varepsilon\geq -\varepsilon\omega$; $\alpha$ is  big  iff it contains a K\"{a}hler current, i.e. a closed $(1,1)$-current $T$ such that $T\geq\varepsilon\omega$ for $\varepsilon>0$ small enough. Finally, $\alpha$ is a  K\"{a}hler class  iff it contains a K\"{a}hler form.
\end{definition}
Since any two Hermitian forms $\omega_1$ and $\omega_2$ are commensurable ( i.e. $C^{-1}\omega_2\leq \omega_1\leq C\omega_2$ for some $C>0$), these definitions do not depend on the choice of $\omega$.

\subsection{Lebesgue decomposition of a current}\label{Lebesgue decomposition} In this subsection, we refer to \cite{Bou02, MM07}.
For a measure $\mu$ on a manifold $M$ we denote by $\mu_{ac}$ and $\mu_{sing}$ the uniquely determined absolute continuous and singular measures (with respect to the Lebesgue measure on $M$) such that
\begin{align*}
\mu=\mu_{ac}+\mu_{sing}
\end{align*}
which is called the Lebesgue decomposition of $\mu$. If $T$ is a $(1,1)$-current of order $0$ on $X$, written locally $T=i\sum T_{ij}dz_i\wedge d\overline{z}_j$, we defines its absolute continuous and singular components by
\begin{align*}
T_{ac}&=i\sum (T_{ij})_{ac}dz_i\wedge d\overline{z}_j,\\
T_{sing}&=i\sum (T_{ij})_{sing}dz_i\wedge d\overline{z}_j.
\end{align*}
The Lebesgue decomposition of $T$ is then
\begin{align*}
T=T_{ac}+T_{sing}. 
\end{align*}
If $T\geq 0$, it follows that $T_{ac}\geq 0$ and $T_{sing}\geq 0$. Moreover, if $T\geq \alpha$ for a continuous $(1,1)$-form $\alpha$, then $T_{ac}\geq \alpha$, $T_{sing}\geq 0$. The Radon-Nikodym theorem insures that $T_{ac}$ is (the current associated to) a $(1,1)$-form with $L^1_{loc}$ coefficients. The form $T_{ac}(x)^n$ exists for almost all $x\in X$ and is denoted $T^n_{ac}$.

Note that $T_{ac}$ in general is not closed, even when $T$ is, so that the decomposition doesn't induce a significant decomposition at the cohomological level. However, when $T$ is a closed positive $(1,1)$-current with analytic singularities along a subscheme $V$, the residual part $R$ in Siu decomposition (c.f.\cite{Siu74}) of $T$ is nothing but $T_{ac}$, and the divisorial part $\sum_k\nu(T,Y_k)[Y_k]$ is $T_{sing}$. The following facts are well-known.
\begin{lemma}[c.f.\cite{Bou02}]\label{push-forward}
Let $f:Y\rightarrow X$ be a proper surjective holomorphic map. If $\alpha$ is a locally integrable form of bidimension $(k,k)$ on $Y$, then the push-forward current $f_*\alpha$ is absolutely continuous, hence a locally integrable form of bidimension $(k,k)$. In particular, when $T$ is a positive current on $Y$, the push-forward current $f_*(T_{ac})$ is absolutely continuous, and we have the formula $f_*(T_{ac})=(f_*T)_{ac}$.

\end{lemma}

% The absolutely continuous part $T_{ac}$ of a positive current $T$ does not depend continuously on $T$, but we have the following semi-continuity property:

%\begin{lem}[c.f.\cite{Bou02}]\label{semi-continuity}
%Let $T_k$ be a sequence of positive $(1,1)$-currents converging weakly to $T$. Then one has
%\begin{align}
%T_{ac}(x)^n\geq \limsup T_{k,ac}(x)^n\notag
%\end{align}
% for almost every $x\in X$.
%\end{lem}
\subsection{Resolution of singularities}\label{resolution of singularities}

\begin{definition}[Currents with analytic singularities]\label{APC}We say that a function $\phi$ on $X$ has analytic singularities along a subscheme $V(\mathscr{I})$ (corresponding to a coherent ideal sheaf $\mathscr{I}$) if there exists $c>0$ such that $\phi$ is locally congruent to $\frac{c}{2}\log(\sum|f_i|^2)$ modulo smooth functions, where $f_1,\cdots, f_r$ are local generators of $\mathscr{I}$.  Note that a function with analytic singularities is automatically almost plurisubharmonic, and is smooth away from the support of $V(\mathscr{I})$.

We say an almost positive $(1,1)$-current  has analytic  singularities, if we can find a smooth form $\theta$ and a function  $\varphi$ on $X$ with analytic   singularities, such that $T=\theta+dd^c\varphi$. Note that one can always write $T=\theta+dd^c\varphi$ with $\theta$ smooth and $\varphi$ almost plurisubharmonic on a compact complex manifold.
\end{definition}

\begin{definition}[Pull back of $(1,1)$-currents]\label{pull-back}
Let $f:Y\rightarrow X$ be a surjective holomorphic map between compact complex manifolds and $T$ be a closed almost positive $(1,1)$-current on $X$. Write $T=\theta+dd^c\varphi$ for some smooth form $\theta\in\{T\}$, and $\varphi$ an almost plurisubharmonic function on $X$.  We define its  pull back $f^*T$ by $f$ to be $f^*\theta+dd^cf^*\varphi$. Note that this definition is independent of the choices made, and we have $\{f^*T\}=f^*\{T\}$.
\end{definition}

We now use the notations in Definition \ref{APC}. From \cite{Hir64, BiM91, BiM97}, one can  blow-up $X$ along $V(\mathscr{I})$ and resolve the singularities, to get a smooth modification $\pi:\widetilde{X}\rightarrow X$, where $\widetilde{X}$ is a compact complex manifold, such that $\pi^{-1}{\mathscr{I}}$ is just $\mathcal{O}(-D)$ for some simple normal crossing divisor $D$ upstairs.   The pull back $\pi^*T$ clearly has analytic singularities along $V(\pi^{-1}(\mathscr{I}))=D$, thus its Siu decomposition  writes

\begin{align}
\pi^*T=(\pi^*T)_{ac}+(\pi^*T)_{sing}=\theta+c[D],\notag
\end{align}
where $\theta$ is a smooth $(1,1)$-form.  If $T\geq \gamma$ for some smooth form $\gamma$, then $\pi^*T\geq \gamma$, and thus $\theta\geq \pi^*\gamma$. We call this operation a resolution of the singularities of $T$.

%---------------------
%regularization
%-----------------
\subsection{Regularization of currents}
The following celebrated regularization  theorem is due to Demailly.
\begin{thm}[c.f.\cite{Dem82, Dem92}]\label{regularization}
Let $T$ be a closed almost positive $(1,1)$-current on a compact Hermitian manifold $(X,\omega)$. Suppose that $T\geq\gamma$ for some smooth $(1,1)$-form $\gamma$ on $X$. Then
There exists a sequence $T_k$ of currents with analytic singularities in $\{T\}$ which converges weakly to $T$, and $T_{k,ac}(x)\rightarrow T_{ac}(x)$ a.e., such that $T_k\geq \gamma-\varepsilon_k\omega$ for some sequence $\varepsilon_k>0$ decreasing to $0$, and such that $\nu(T_k,x)$ increases to $\nu(T,x)$ uniformly with respect to $x\in X$.
\end{thm}

%---------------
%Lamari's criterion
%------------------
\subsection{Lamari's criterion}\label{lamari}
\begin{thm}[c.f. \cite{Lam99a}]\label{la-lemma}
Let $X$ be an $n$-dimensional compact complex manifold and let $\Phi$ be a real $(k,k)$-form, then there exists a real $(k-1,k-1)$-current $\Psi$ such that $\Phi+dd^c\Psi$ is positive iff for any strictly positive $\partial\overline{\partial}$-closed $(n-k,n-k)$-forms $\Upsilon$, we have $\int_X\Phi\wedge\Upsilon\geq 0$.
\end{thm}

\begin{rmk}
Theorem \ref{la-lemma} also holds for  closed positive $(1,1)$-currents. 
\end{rmk}

%--------------------
%Complex Monge-Ampere equation
%-----------------
\subsection{Complex Monge-Amp\`{e}re equations}
%The following theorem on degenerate complex Monge-Amp\`{e}re equations is due to Demailly-Pali \cite{DPl10}.
%\begin{thm}[\cite{DPl10}] \label{DPl-thm}
%	Let $X$ be a compact connected K\"{a}hler manifold of complex dimension $n$, let $\Omega$ be a smooth volume form, let $\gamma$ be a closed positive $(1,1)$-current with continuous local potentials with $\int_X\gamma^n>0$. Let also $\psi\in Psh(X,\gamma)\cap L^\infty(X)$ be a solution of the degenerate complex Monge-Amp\`{e}re equation 
%	\begin{align*}
%	(\gamma+dd^c\psi)^n=f\Omega,
%	\end{align*}
%with $f\in L\log^{n+\varepsilon_0}L(X)$ for some $\varepsilon_0>0$. Then there exists a uniform constant $C_1=C_1(\varepsilon_0,\gamma,\Omega)$ 
%\end{thm}
%

The following theorem on complex Monge-Amp\`{e}re equations on compact Hermitian manifolds  is due to Tosatti-Weinkove \cite{TW10}.
\begin{thm}[\cite{TW10}]\label{Monge-Ampere equation}
	Let $(X,\omega)$ be a compact Hermitian manifold. For any smooth real-valued function $F$ on $X$, there exist a unique real number $C>0$ and a unique smooth real-valued function $\phi$ on $X$ solving
	\begin{align}
	(\omega+i\partial\overline{\partial}\phi)^n=Ce^F\omega^n,\notag
	\end{align}
	with $\omega+i\partial\overline{\partial}\phi>0$ and $\sup_X\phi=0$.  Furthermore, if $\partial\overline{\partial}\omega^k=0$ for $1\leq k\leq n-1$, then we have
	\begin{align}
	C=\frac{\int_X\omega^n}{\int_Xe^F\omega^n}.\notag
	\end{align}
\end{thm}

%--------------------
% A simple proof of Theorem  Wang
%-------------------
\section{A simplified   proof of Theorem \ref{Wang}}

From the definition of vol$(\alpha)$, one can find a positive closed current $S\in \alpha$ such that $\int_XS^n_{ac}>\frac{\mbox{vol}(\alpha)}{2}>0$. Then by Demailly's regularization theorem (Theorem \ref{regularization}), combined with Fatou's lemma, we can find a sequence $T_k$ of closed currents with analytic singularities in $\alpha$ such that 
\begin{itemize}
\item $T_k\geq -\varepsilon_k\omega$, where $\varepsilon_k\searrow 0$  as $k\rightarrow \infty$;
\item $\int_XT^n_{k,ac}\geq c>0$.
\end{itemize}

For each $k$, we choose a smooth proper modification $\mu_k:X_k\rightarrow X$ such that 
\begin{itemize}
\item  $\mu^*_kT_k=\theta_k+D_k$, where $\theta_k\geq -\varepsilon_k\mu_k^*\omega$ is a smooth closed form and $D_k$ is a real effective divisor.
\item $\int_XT^n_{k,ac}=\int_{X_k}\theta_k^n\geq c>0$.
\end{itemize}

\begin{lemma}[c.f. \cite{Dem12}]\label{blow-up}Suppose that $(X,\omega)$ is a compact Hermitian manifold with a Hermitian metric $\omega$ satisfying $\partial\overline{\partial}\omega^k=0$ for all $k$.  Let $\pi:\widetilde{X}\rightarrow X$ be a smooth modification (a tower of blow-ups).  Then  there exists a Hermitian metric $\Omega$ on $\widetilde{X}$ such that
	\begin{itemize} 
		\item $\partial\overline{\partial}\Omega^k=0$  for all $k$ on $\widetilde{X}$;
		\item for any give $\varepsilon>0$, one can make the inequality $\|\Omega-\mu^*\omega\|<\varepsilon$ holds.
		\end{itemize}
\end{lemma}

\begin{proof}
Suppose that $\widetilde{X}$ is obtained as a tower of blow-ups
\begin{align}
\widetilde{X}=X_N\rightarrow X_{N-1}\rightarrow \cdots\rightarrow X_1\rightarrow X_0=X,\notag
\end{align}
where $X_{j+1}$ is the blow-up of $X_j$ along a smooth center $Y_j\subset X_j$. Denote by $E_{j+1}\subset X_{j+1}$ the exceptional divisor, and let $\pi_j:X_{j+1}\rightarrow X_j$ be the blow-up map. The line bundle $\mathcal{O}(-E_{j+1})|_{E_{j+1}}$ is equal to $\mathcal{O}_{P(N_j)}(1)$ where $N_j$ is the normal bundle to $Y_j$ in $X_j$. Pick an arbitrary smooth Hermitian metric on $N_j$, use this metric to get an induced Fubini-Study metric on $\mathcal{O}_{P(N_j)}(1)$, and finally extend this metric as a smooth Hermitian metric on the line bundle $\mathcal{O}(-E_{j+1})$. Such a metric has positive curvature along tangent vectors of $X_{j+1}$ which are tangent to the fibers of $E_{j+1}=P(N_j)\rightarrow Y_j$. Assume further that $\omega_j$ is a Gauduchon metric satisfying assumption (*) on $X_j$. Then
\begin{align}
\Omega_{j+1}=\pi^*_j\omega_j-\varepsilon_{j+1}u_{j+1}\notag
\end{align}
where $\pi^*_j\omega_j$ is  semi-positive on $X_{j+1}$, positive definite on $X_{j+1}\setminus E_{j+1}$, and also positive definite on tangent vectors of $T_{X_{j+1}}|_{E_{j+1}}$ which are not tangent to the fibers of $E_{j+1}\rightarrow Y_j$. It  is then easily to see that $\Omega_{j+1}>0$   by taking $\varepsilon_{j+1}\ll 1$. Thus our final candidate $\Omega$ on $\widetilde{X}$ has the form $\Omega=\pi^*\omega-\sum\varepsilon_j\widetilde{u}_j$, where $\widetilde{u}_j=(\pi_{N-1}\circ \cdots\circ \pi_{j})^*u_j$. Since every $u_j$ is a curvature term of a line bundle, the term $\sum\varepsilon_j\widetilde{u}_j$ is $d$-closed, thus $\partial\overline{\partial}\Omega^k=0$. Furthermore, by choosing $\varepsilon_j$ sufficienlty small, one can make $\|\Omega-\mu^*\omega\|<\varepsilon$  for any give $\varepsilon>0$.
\end{proof}

Select on each $X_k$ a Gauduchon metric $\widetilde{\omega}_k$ which satisfies $\partial\overline{\partial}\widetilde{\omega}_k^l=0$ for all $l$, and $\|\widetilde{\omega}_k-\mu_k^*\omega\|<\varepsilon$  for any give $\varepsilon>0$ .

In the sequel, we will show that for $k$ large, the class ${\theta_k}$ is big on $X_k$. 

We argue by contradiction. By applying Lamari's criterion (Theorem \ref{lamari}), and suppose to the contrary, for any $k$, and any $m\in \mathbb{N}$, there is a Gauduchon metric $\omega_{k,m}$ on $X_k$ such that 
\begin{align*}
\int_{X_k}\theta_k\wedge \omega^{n-1}_{k,m}\leq \frac{1}{m}\int_{X_k}\widetilde{\omega}_k\wedge \omega^{n-1}_{k,m}.
\end{align*}

Without loss of generality, we assume that 
\begin{align*}
\int_{X_k}\widetilde{\omega}_k\wedge \omega^{n-1}_{k,m}=1
\end{align*}
and therefore 
\begin{align*}
\int_{X_k}\theta_k\wedge\omega^{n-1}_{k,m}\leq \frac{1}{m}.
\end{align*}

From Theorem \ref{Monge-Ampere equation}, we solve the following equation on $X_k$,
\begin{align*}
(\theta_k+\varepsilon_k\mu^*_k\omega+\frac{1}{m}\widetilde{\omega}_{k}+dd^c\varphi_{k,m})^n=C_{k,m}\omega^{n-1}_{k,m}\wedge\widetilde{\omega}_k.
\end{align*}

Set $\alpha_{k,m}=\theta_k+\varepsilon_k\mu_k^*\omega+\frac{1}{m}\widetilde{\omega}_k+dd^c\varphi_m$, then $\alpha_{k,m}>0$. 
We need the following 
\begin{lemma}[c.f. \cite{N16, Pop16}]\label{key lemma}
For every $k, m\in \mathbb{N}^*$, we have that 
\begin{align*}
\big(\int_{X_k}{\alpha}_{k,m}\wedge \omega^{n-1}_{k,m}\big)\big(\int_{X_k}\alpha_{k,m}^{n-1}\wedge \widetilde{\omega}_k\big)&\geq \frac{1}{n}\Big(\int_{X_k}\sqrt{\frac{{\alpha}^n_{k,m}}{\widetilde{\omega}_k^n}\frac{\widetilde{\omega}_k\wedge\omega_{k,m}^{n-1}}{\widetilde{\omega}_k^n}}\widetilde{\omega}_k^n\Big)^2\\&=\frac{C_{k,m}}{n}.
\end{align*}
\end{lemma}

Firstly, we have that 
\begin{align}\label{estimate of Ckm}
C_{k,m}&=\int_{X_k}(\theta_k+\varepsilon_k\mu^*_k\omega+\frac{1}{m}\widetilde{\omega}_{k}+dd^c\varphi_{k,m})^n\\
&=\int_{X_k}(\theta_k+\varepsilon_k\mu^*_k\omega+\frac{1}{m}\widetilde{\omega}_{k})^n\geq\int_{X_k}(\theta_k+\varepsilon_k\mu^*_k\omega)^n\notag\\
&=\int_{X}(T_{k,ac}+\varepsilon_k\omega)^n=\int_XT_{k,ac}^n+O(\varepsilon_k)\notag\\
&\geq C>0\notag
\end{align}
for $k$ sufficiently large, where $C$ is a uniform constant independent of $k$ and $m$.

Secondly, we have the following observation
\begin{align}\label{estimate of the second term}
\int_{X_k}\alpha_{k,m}^{n-1}\wedge \widetilde{\omega}_k&=\int_{X_k}(\theta_k+\varepsilon_k\mu^*_k\omega+\frac{1}{m}\widetilde{\omega}_{k}+dd^c\varphi_{k,m})^{n-1}\wedge \widetilde{\omega}_{k}\\
&=\int_{X_k}(\theta_k+\varepsilon_k\mu^*_k\omega+\frac{1}{m}\widetilde{\omega}_{k})^{n-1}\wedge \widetilde{\omega}_{k}\notag\\
&\leq \int_{X}(T_{k,ac}+\varepsilon_k\omega+\frac{2}{m}\omega)^{n-1}\wedge (2\omega)\leq M (>0)\notag
\end{align}
for $k$ sufficiently large, where $M$ is a uniform constant independent of $k$ and $m$.

Thirdly, we see that for $k$ large
\begin{align}\label{estimate of the last term}
\int_{X_k}{\alpha}_{k,m}\wedge \omega^{n-1}_{k,m}&=\int_{X_k}(\theta_k+\varepsilon_k\mu^*_k\omega+\frac{1}{m}\widetilde{\omega}_{k}+dd^c\varphi_{k,m})\wedge \omega^{n-1}_{k,m}\\
&=\int_{X_k}(\theta_k+\varepsilon_k\mu^*_k\omega+\frac{1}{m}\widetilde{\omega}_{k})\wedge \omega^{n-1}_{k,m}\notag
\\
&\leq\int_{X_k}\theta_k\wedge \omega^{n-1}_{k,m}+\frac{2}{m}\int_{X_k}\widetilde{\omega}_k\wedge\omega^{n-1}_{k,m}\leq \frac{3}{m}.\notag
\end{align}

Finally, from Lemma \ref{key lemma}, (\ref{estimate of Ckm}), (\ref{estimate of the second term}) and (\ref{estimate of the last term}), we obtain that  for $k$ large
\begin{align*}
\frac{3M}{m}\geq \frac{C}{n},
\end{align*}
which is absurd for $m>>0$. 

Now we arrive at  the fact that, for $k$ large, $\{\theta_k\}$ is a big class on $X_k$.  Take a K\"{a}hler current $\Theta$ in $\{\theta_k\}$, one can get a K\"{a}hler current $(\mu_k)_*(\Theta+ D_k)$ on $X$, which belongs to the class $\alpha$.  Thus $X$ is in the Fujiki class $\mathcal{C}$, which will force $X$ to be K\"{a}hler, since it supports a $\partial\overline{\partial}$-closed Hermitian metric. Thus we complete the proof of Theorem \ref{Wang}.

%-------------------------------------------------------------
%Proof of the main result
%------------------------------------------------------------
\section{Proof of  Theorem \ref{main theorem}}
Now let $\pi:\widetilde{X}\rightarrow X$ be the resolution of singularities such that $\pi^{-1}{\mathscr{I}}$ is just $\mathcal{O}(-D)$ for some simple normal crossing divisor $D$ upstairs. Thus we have
\begin{align}
\pi^*T=(\pi^*T)_{ac}+(\pi^*T)_{sing}=\theta+c[D].\notag
\end{align}

Now we are on the way to prove Theorem \ref{main theorem}.  Since $\int_X T^n_{ac}>0$, we have that
\begin{align*}
\int_{\widetilde{X}}\theta^n=\int_{\widetilde{X}}(\pi^*T)^n_{ac}=\int_XT^n_{ac} >0,
\end{align*}
where second equality follows from Lemma \ref{push-forward}.  

From Lemma \ref{blow-up}, there is a Hermitian metric $\Omega$ on $\widetilde{X}$ such that $\partial\overline{\partial}\Omega=0$. 

Since $\theta$ is smooth and semi-positive,   $\int_{\widetilde{X}}\theta^n>0$,   from Nguyen's Theorem \ref{Ngu}, one can conclude that the class $[\theta]$ is big. Thus the class $[\pi_*(\pi^*T)]=[T]=[\pi_*(\theta+c[D])]$ is big. So $X$ is in the Fujiki class $\mathcal{C}$, and finally K\"{a}hler for the same reason as above,  which completes the  proof of  Theorem \ref{main theorem}.


\begin{thebibliography}{10}
	
	\bibitem{BiM91}
	E.~Bierstone and P.~D. Milman.
	\newblock A simple constructive proof of canonical resolution of singularities.
	\newblock In {\em Effective methods in algebraic geometry ({C}astiglioncello,
		1990)}, volume~94 of {\em Progr. Math.}, pages 11--30. Birkh\"auser Boston,
	Boston, MA, 1991.
	
	\bibitem{BiM97}
	E.~Bierstone and P.~D. Milman.
	\newblock Canonical desingularization in characteristic zero by blowing up the
	maximum strata of a local invariant.
	\newblock {\em Invent. Math.}, 128(2):207--302, 1997.
	
	\bibitem{Bon98}
	L.~Bonavero.
	\newblock In\'egalit\'es de morse holomorphes singuli\`eres.
	\newblock {\em J. Geom. Anal.}, 8(3):409--425, 1998.
	
	\bibitem{Bou02}
	S.~Boucksom.
	\newblock On the volume of a line bundle.
	\newblock {\em Internat. J. Math.}, 13(10):1043--1063, 2002.
	
	\bibitem{BDPP}
	S.~Boucksom, J.-P. Demailly, M.~P\u{a}un, and T.~Peternell.
	\newblock The pseudo-effective cone of a compact {K}\"ahler manifold and
	varieties of negative {K}odaira dimension.
	\newblock {\em J. Algebraic Geom.}, 22(2):201--248, 2013.
	
	\bibitem{Chi16}
	I.~Chiose.
	\newblock The {K}\"ahler rank of compact complex manifolds.
	\newblock {\em J. Geom. Anal.}, 26(1):603--615, 2016.
	
	\bibitem{Dem82}
	J.-P. Demailly.
	\newblock Estimations {$L^{2}$}\ pour l'op\'erateur {$\bar \partial $}\ d'un
	fibr\'e vectoriel holomorphe semi-positif au-dessus d'une vari\'et\'e
	k\"ahl\'erienne compl\`ete.
	\newblock {\em Ann. Sci. \'Ecole Norm. Sup. (4)}, 15(3):457--511, 1982.
	
	\bibitem{Dem85}
	J.-P. Demailly.
	\newblock Champs magn\'etiques et in\'egalit\'es de {M}orse pour la
	{$d''$}-cohomologie.
	\newblock {\em Ann. Inst. Fourier (Grenoble)}, 35(4):189--229, 1985.
	
	\bibitem{Dem91}
	J.-P. Demailly.
	\newblock Holomorphic {M}orse inequalities.
	\newblock In {\em Several complex variables and complex geometry, {P}art 2
		({S}anta {C}ruz, {CA}, 1989)}, volume~52 of {\em Proc. Sympos. Pure Math.},
	pages 93--114. Amer. Math. Soc., Providence, RI, 1991.
	
	\bibitem{Dem92}
	J.-P. Demailly.
	\newblock Regularization of closed positive currents and intersection theory.
	\newblock {\em J. Algebraic Geom.}, 1(3):361--409, 1992.
	
	\bibitem{Dem12}
	J.-P. Demailly.
	\newblock {\em Analytic methods in algebraic geometry}, volume~1 of {\em
		Surveys of Modern Mathematics}.
	\newblock International Press, Somerville, MA; Higher Education Press, Beijing,
	2012.
	
	\bibitem{DP04}
	J.-P. Demailly and M.~P\u{a}un.
	\newblock Numerical characterization of the {K}\"ahler cone of a compact
	{K}\"ahler manifold.
	\newblock {\em Ann. of Math. (2)}, 159(3):1247--1274, 2004.
	
	\bibitem{GR70}
	H.~Grauert and O.~Riemenschneider.
	\newblock Verschwindungss\"atze f\"ur analytische {K}ohomologiegruppen auf
	komplexen {R}\"aumen.
	\newblock {\em Invent. Math.}, 11:263--292, 1970.
	
	\bibitem{Hir64}
	H.~Hironaka.
	\newblock Resolution of singularities of an algebraic variety over a field of
	characteristic zero. {I}, {II}.
	\newblock {\em Ann. of Math. (2) {\bf 79} (1964), 109--203; ibid. (2)},
	79:205--326, 1964.
	
	\bibitem{JS93}
	S.~Ji and B.~Shiffman.
	\newblock Properties of compact complex manifolds carrying closed positive
	currents.
	\newblock {\em J. Geom. Anal.}, 3(1):37--61, 1993.
	
	\bibitem{Lam99a}
	A.~Lamari.
	\newblock Courants k\"ahl\'eriens et surfaces compactes.
	\newblock {\em Ann. Inst. Fourier (Grenoble)}, 49(1):vii, x, 263--285, 1999.
	
	\bibitem{MM07}
	X.~Ma and G.~Marinescu.
	\newblock {\em Holomorphic {M}orse inequalities and {B}ergman kernels}, volume
	254 of {\em Progress in Mathematics}.
	\newblock Birkh\"auser Verlag, Basel, 2007.
	
	\bibitem{N16}
	N.~C. Nguyen.
	\newblock The complex {M}onge-{A}mp\`ere type equation on compact {H}ermitian
	manifolds and applications.
	\newblock {\em Adv. Math.}, 286:240--285, 2016.
	
	\bibitem{Pop08}
	D.~Popovici.
	\newblock Regularization of currents with mass control and singular {M}orse
	inequalities.
	\newblock {\em J. Differential Geom.}, 80(2):281--326, 2008.
	
	\bibitem{Pop16}
	D.~Popovici.
	\newblock Sufficient bigness criterion for differences of two nef classes.
	\newblock {\em Math. Ann.}, 364(1-2):649--655, 2016.
	
	\bibitem{Siu74}
	Y.~T. Siu.
	\newblock Analyticity of sets associated to {L}elong numbers and the extension
	of closed positive currents.
	\newblock {\em Invent. Math.}, 27:53--156, 1974.
	
	\bibitem{Siu84b}
	Y.~T. Siu.
	\newblock A vanishing theorem for semipositive line bundles over non-{K}\"ahler
	manifolds.
	\newblock {\em J. Differential Geom.}, 19(2):431--452, 1984.
	
	\bibitem{Tos16}
	V.~Tosatti.
	\newblock The {C}alabi-{Y}au theorem and {K}\"ahler currents.
	\newblock {\em Adv. Theor. Math. Phys.}, 20(2):381--404, 2016.
	
	\bibitem{TW10}
	V.~Tosatti and B.~Weinkove.
	\newblock The complex {M}onge-{A}mp\`ere equation on compact {H}ermitian
	manifolds.
	\newblock {\em J. Amer. Math. Soc.}, 23(4):1187--1195, 2010.
	
	\bibitem{W16}
	Z.~Wang.
	\newblock On the volume of a pseudo-effective class and semi-positive
	properties of the {H}arder-{N}arasimhan filtration on a compact {H}ermitian
	manifold.
	\newblock {\em Ann. Polon. Math.}, 117(1):41--58, 2016.
	
\end{thebibliography}
\end{document}